\documentclass[a4paper, 10pt, twoside, notitlepage]{amsart}

\usepackage{amsmath,amscd}
\usepackage{amssymb}
\usepackage{amsthm}
\usepackage{comment}
\usepackage{graphicx, xcolor}

\usepackage{mathrsfs}
\usepackage[ocgcolorlinks, linkcolor=blue]{hyperref}

\usepackage{bm}
\usepackage{bbm}
\usepackage{url}

\usepackage[utf8]{inputenc}
\usepackage{mathtools,amssymb}
\usepackage{esint}
\usepackage{tikz}
\usepackage{dsfont}
\usepackage{relsize}
\usepackage{url}
\usepackage{xcolor}
\usepackage{graphicx}
\usepackage{mathrsfs}
\usepackage[shortlabels]{enumitem}
\usepackage{lineno}
\usepackage{amsmath}
\usepackage{enumitem}
\usepackage{amsthm} 
\usepackage{verbatim}
\usepackage{dsfont}
\numberwithin{equation}{section}

\allowdisplaybreaks

\mathtoolsset{showonlyrefs}

\graphicspath{{images/}}

\newtheorem{theorem}{Theorem}[section]
\newtheorem{lemma}[theorem]{Lemma}

\newtheorem{corollary}[theorem]{Corollary}

\newtheorem{remark}[theorem]{Remark}

\title[Monotonicity in nonlocal elliptic equations]{Monotonicity and local uniqueness for an isotropic nonlocal elliptic equation}

\author[Y.-H.~Lin]{Yi-Hsuan Lin}
\address{Department of Applied Mathematics, National Yang Ming Chiao Tung University, Hsinchu, Taiwan \& Fakult{\"a}t f{\"u}r Mathematik, University of Duisburg-Essen, Essen, Germany}
\email{yihsuanlin3@gmail.com}


\newcommand{\R}{{\mathbb R}}

\newcommand{\N}{{\mathbb N}}

\newcommand {\p} {\partial}

\newcommand{\LC}{\left(}
\newcommand{\RC}{\right)}
\newcommand{\wt}{\widetilde}

\newcommand{\norm}[1]{\lVert #1 \rVert}
\newcommand{\abs}[1]{\left\lvert #1 \right\rvert}%absolute value
\DeclareMathOperator{\supp}{supp} %support
\begin{document}

	\maketitle
	\begin{abstract}
		We extend monotonicity-based inversion methods to an inverse coefficient problem for the isotropic nonlocal elliptic equation
		\[
		(-\nabla \cdot \sigma \nabla)^s u = 0 \quad \text{in } \Omega \subset \mathbb{R}^n,
		\]
		where \(0 < s < 1\), \(n \geq 3\), and \(\Omega\) is a bounded open set. We establish a monotonicity relation between the leading coefficient \(\sigma\) and the (partial) exterior Dirichlet-to-Neumann (DN) map. Our main result shows that a monotonicity ordering of the coefficients implies a corresponding ordering of the DN maps. Furthermore, we construct localized potentials for the nonlocal equation, which yield a local uniqueness result for the fractional inverse problem.

	\end{abstract}

	\tableofcontents
	
	\section{Introduction}\label{sec: introduction}

Inverse problems for nonlocal operators have attracted considerable attention in recent years, with the fractional Laplacian $(-\Delta)^s$ ($0 < s < 1$) being a key example. This operator emerges naturally in models of anomalous stochastic diffusion, characterized by jumps and long-range interactions, as explored in works such as \cite{BV_16, Ros-Oton_16}. In contrast to the classical Laplacian ($s = 1$), which describes standard Brownian motion, the nonlocal nature of the fractional Laplacian introduces significant complexity. Nevertheless, recent progress indicates that inverse problems for nonlocal equations may be more manageable than their local counterparts, thanks to robust properties like unique continuation and Runge approximation.

% Discussing the fractional Calderón problem and key properties
The Calderón problem for the fractional Schrödinger equation was first investigated in \cite{GSU20}. Central to this work is the Runge approximation property, which allows any $L^2$ function to be approximated by solutions to the fractional Schrödinger equation. This property derives from the unique continuation principle (UCP), which asserts that if $u = (-\Delta)^s u = 0$ in any nonempty open subset, then $u \equiv 0$ across $\mathbb{R}^n$. Subsequent research has extended these findings to variable-coefficient nonlocal elliptic operators \cite{GLX}, addressing a challenge that remains unresolved for their local counterparts.

% Highlighting the development of related literature
Building on these foundations, a substantial body of literature has developed. Simultaneous determination results are established in \cite{CLL2017simultaneously, cekic2020calderon, LL2022inverse}, while stability estimates are derived in \cite{RS17, KLW2021calder, ruland2018exponential}. Additionally, monotonicity-based methods for nonlocal inverse problems have been introduced in \cite{HL19_monotone1, HL20_monotone2, lin2020monotonicity}. Further contributions, encompassing both linear and nonlinear settings, are detailed in \cite{LL2020inverse, GRSU20, CMRU20, LZ2023unique} and references therein.

% Outlining recent innovative approaches
Recent advancements have introduced innovative approaches to inverse problems for nonlocal operators. The recovery of leading coefficients has been achieved through novel nonlocal-to-local reductions utilizing the Caffarelli--Silvestre extension, as demonstrated in \cite{CGRU2023reduction, ruland2023revisiting, LLU2023calder, LZ2024approximation}. Additionally, local-to-nonlocal reduction is characterized in the transversally anisotropic setting in \cite{LNZ_Calderon}. Meanwhile, heat semigroup methods on closed Riemannian manifolds have proven effective for addressing the fractional anisotropic Calderón problem, as explored in \cite{feizmohammadi2021fractional, Fei24_TAMS, FKU2024calder, lin2024fractional, FGKRSU25}. Furthermore, the entanglement principle for nonlocal elliptic operators, investigated in \cite{FKU2024calder, FL24}, shows promise for analyzing systems of nonlocal equations. The Calderón problem for the logarithmic Laplacian—a zero-order nonlocal operator—has also been recently addressed in \cite{HLW_2024_log}. For a comprehensive overview of inverse problems for nonlocal operators, we refer readers to the recent monograph \cite{LL25_IPBook}.\\

	\noindent \textbf{Mathematical formulation.} Let $\Omega\subset \R^n$ be a bounded Lipschitz domain, for $n\geq 3$, and $0<s<1$. Consider the exterior value problem
	\begin{equation}\label{equ: main}
		\begin{cases}
			(-\nabla \cdot \sigma\nabla)^s u =0 &\text{ in }\Omega, \\
			u=f & \text{ in }\Omega_e,
		\end{cases}
	\end{equation}
	where $\sigma=\sigma(x)\in C^2(\R^n)$ satisfying 
	\begin{equation}\label{ellipticity}
		0< \lambda \leq \sigma(x) \leq \lambda^{-1}  \text{ for }x\in \overline{\Omega} \quad \text{and} \quad \sigma=1 \text{ in }\Omega_e,
	\end{equation}
	for some $\lambda\in (0,1)$, and 
	$$
	\Omega_e:=\R^n\setminus \overline{\Omega}
	$$ 
	stands for the exterior domain. Throughout this work, we assume that the condition \eqref{ellipticity} always holds.
	Let $W\subset \Omega_e$ be a bounded open set with $\overline{\Omega}\cap \overline{W}=\emptyset$, and $\Lambda_\sigma$ be the \emph{Dirichlet-to-Neumann} (DN) map of \eqref{equ: main}, which is given by 
	\begin{equation}\label{DN map}
		\Lambda_\sigma : \wt H^s(W) \to H^{-s}(W), \quad f\mapsto \left. (-\nabla \cdot \sigma \nabla )^s u^f \right|_{W},
	\end{equation}
	where $u^f\in H^s(\R^n)$ is the solution to \eqref{equ: main}.
	In this work, we are interested in a monotonicity relation between the DN map and leading coefficients $\sigma$. Throughout this work, let us assume that $\sigma$ is a positive bounded scalar function, with $\sigma|_{\Omega_e}$ being known a priori.
	
	In the works \cite{HL19_monotone1,HL20_monotone2}, the authors demonstrated if-and-only-if monotonicity relations between the DN maps with lower order bounded potentials. To our best knowledge, there is no existing literature to consider such relations between the DN maps with leading coefficients in \eqref{equ: main}.
	In fact, in Section \ref{sec: mono}, we are going to prove that 
	\begin{equation}
		\sigma_1\geq \sigma_2 \text{ in }\Omega \implies \Lambda_{\sigma_1}\geq \Lambda_{\sigma_2}.
	\end{equation}
	Let us make the above sentence meaningful. On the one hand, the relation $\sigma_1\geq \sigma_2$ in $\Omega$ is referred to $\sigma_1(x)\geq \sigma_2(x)$ for almost every (a.e.) $x\in \Omega$. In this work, we have assumed that $\sigma_1,\sigma_2$ satisfy \eqref{ellipticity}, so we justify the half-ordering $\sigma_1\geq \sigma_2$ is referred to $\sigma_1(x)\geq \sigma_2(x)$, for all $x\in \Omega$. On the other hand, we write $\Lambda_{\sigma_1}\geq \Lambda_{\sigma_2}$, if it holds in the quadratic sense that 
	\begin{equation}
		\left\langle \LC \Lambda_{\sigma_1}-\Lambda_{\sigma_2}\RC f , f \right\rangle \geq0,
	\end{equation}
	for any $f\in C^\infty_c(W)$. Here, $\langle \cdot , \cdot \rangle$ denotes the duality pairing in a suitable sense (see Section \ref{sec: prel}).

	The method to study inverse problems using the combination of monotonicity relations with localized potentials was found in \cite{gebauer2008localized}. Due to this remarkable approach and the flexibility of this method, there is some literature in this direction, \cite{arnold2013unique,harrach2009uniqueness,harrach2010exact,harrach2012simultaneous,harrach2013monotonicity,barth2017detecting,harrach2017local,brander2018monotonicity,griesmaier2018monotonicity,harrach2018helmholtz,harrach2018localizing,seo2018learning,harrach2019dimension}. In further, several works build practical reconstruction methods based on monotonicity properties \cite{tamburrino2002new,harrach2015combining,harrach2015resolution,harrach2016enhancing,maffucci2016novel,tamburrino2016monotonicity,garde2017comparison,garde2017convergence,su2017monotonicity,ventre2017design,harrach2018monotonicity,zhou2018monotonicity,garde2019regularized}.

	We revisit the fractional Calder\'on problem for the isotropic nonlocal elliptic equation \eqref{equ: main}. By \cite{CGRU2023reduction}, as $\sigma$ is a scalar function fulfilling \eqref{ellipticity}, the authors proved the global uniqueness for \eqref{equ: main}. In other words, the isotropic scalar function $\sigma$ can be determined uniquely by the DN map $\Lambda_\sigma$. In this work, we prove a local uniqueness result of \eqref{equ: main}:
	
	\begin{theorem}[Local uniqueness]\label{thm: local}
		Let $\Omega \subset \R^n$ be a bounded domain with Lipschitz boundary $\p \Omega$ for $n\geq 3$ and $W\Subset \Omega_e$ be a nonempty open subset. Let $\mathcal{O}\subseteq \overline{\Omega}$ be a connected relatively open subset such that $\mathcal{O}\cap \p \Omega\neq \emptyset$. Let $\sigma_j \in C^2(\R^n)$ satisfy \eqref{ellipticity}, and $\Lambda_{\sigma_j}$ be the DN map of 	\begin{equation}\label{equ: main j=1,2}
			\begin{cases}
				(-\nabla \cdot \sigma_j\nabla)^s u_j=0 &\text{ in }\Omega, \\
				u_j=f & \text{ in }\Omega_e,
			\end{cases}
		\end{equation}
		for $j=1,2$. Suppose
		\begin{equation}\label{pointwise mono assumption}
			\begin{split}
				\text{either}\quad \sigma_1\leq \sigma_2 \text{ in }\mathcal{O} \quad \text{or}\quad \sigma_1 \geq \sigma_2 \text{ in }\mathcal{O},
			\end{split}
		\end{equation}
		then
		\begin{equation}\label{equal DN map}
			\left. \Lambda_{\sigma_1}  f  \right|_W= \left. \Lambda_{\sigma_2} f \right|_W, \quad \text{for any }f\in C^\infty_c(W),
		\end{equation}
		implies $\sigma_1=\sigma_2$ in $\mathcal{O}$.
	\end{theorem}

	\begin{remark}
		There is an alternative way to show that the nonlocal DN maps $\Lambda_\sigma$ of \eqref{equ: main} determine their local DN maps of 
		\begin{equation}
			\begin{cases}
				\nabla \cdot \LC \sigma \nabla v \RC =0  &\text{ in }\Omega, \\
				v=g \in H^{1/2}(\p \Omega) &\text{ on }\p \Omega,
			\end{cases}
		\end{equation} 
		whenever the condition \eqref{ellipticity} holds.
		Therefore, combining $\sigma|_{\p \Omega}=1$, one can use the well-known result from \cite{sylvester1987global} so that the scalar conductivity $\sigma$ can be determined in the entire domain $\Omega$. In other words, a global uniqueness result for \eqref{equ: main} can be derived by using this nonlocal-to-local reduction procedure.
	\end{remark}

	\noindent \textbf{Organization of the article.} The structure of the paper is as follows. In Section~\ref{sec: prel}, we introduce the function spaces, nonlocal operators, and extension problems that will be used throughout the article. Section~\ref{sec: mono} is devoted to proving the monotonicity relation between the DN maps and the leading coefficients. A key step in this analysis is the construction of localized potentials for the extension problem, which is carried out in Section~\ref{sec: localized potential}. Finally, in Section~\ref{sec: thm}, we combine these ingredients to establish our main results.

	\section{Preliminaries}\label{sec: prel}

	\subsection{Function spaces}

	Let us quickly review some function spaces together with the definition of the operator $(-\nabla \cdot \sigma\nabla)^s$, which are introduced in many related articles. Given $0<s<1$, the space $H^s(\R^n)=W^{s,2}(\R^n)$ denotes the usual $L^2$-based fractional Sobolev space with the given norm 
	\begin{equation*}
		\|u\|_{H^{s}(\mathbb{R}^{n})}:=\|u\|_{L^{2}(\R^{n})}+[u]_{H^{s}(\R^{n})}\label{eq:NormHs}
	\end{equation*}
	where $[\cdot]_{H^s(B)}$ 
	\[
	[u]_{H^{s}(B)}:=\LC\int_{B\times B}\frac{\left|u(x)-u(y)\right|^{2}}{|x-y|^{n+2s}}\, dxdy\RC^{1/2},
	\]
	is the seminorm, for any open set $B\subset \R^{n}$.
	
	Motivated by \cite{GSU20}, let $B\subset \R^n$ be a nonempty bounded open set with Lipschitz boundary, $C_{c}^{\infty}(B)$ contains all $C^{\infty}(\mathbb{R}^{n})$-smooth functions supported in $B$. Given $b\in \R$, let us adopt the following notions
	\begin{align*}
		H^{b}(B) & :=\left\{u|_{B}: \, u\in H^{b}(\R^{n})\right\},\\
		\wt H^{b}(B) & :=\text{closure of \ensuremath{C_{c}^{\infty}(B)} in \ensuremath{H^{b}(\R^{n})}},\\
		H_{0}^{b}(B) & :=\text{closure of \ensuremath{C_{c}^{\infty}(B)} in \ensuremath{H^{b}(B)}},
	\end{align*}
	for different fractional Sobolev spaces.
	$H^{b}(B)$ is complete in the sense
	\[
	\|u\|_{H^{b}(B)}:=\inf\left\{ \|w\|_{H^{b}(\mathbb{R}^{n})}: \, w\in H^{b}(\mathbb{R}^{n})\mbox{ and }w|_{B}=u\right\} .
	\]
	As the exponent $b=s\in (0,1)$, $H^{-s}(B)$ stands for the dual space of $\wt H^s(B)$, so that $H^{-s}(B)$ can be characterized by  
	\[
	H^{-s}(B) = \left\{ u|_{B} : u \in H^{-s}(\mathbb{R}^{n}) \right\} \quad \text{with} \quad \inf_{w\in H^{s}(\mathbb{R}^{n}), \ w|_{B}=u} \| w \|_{H^{s}(\mathbb{R}^{n})},
	\]
	In addition, we always denote 
	\begin{equation}
		\big( \wt H^s(B)\big)^\ast = H^{-s}(B)  \quad \text{and}\quad \big( H^s(B) \big)^\ast = \wt H^{-s}(B).
	\end{equation}
	throughout this paper.

	Moreover, we also introduce $L^2$-weighted  Sobolev spaces for the Caffarelli-Silvestre extension problem. Let $A\subseteq \mathbb{R}^{n+1}_+$ be a nonempty set, $y>0$, and consider $L^{2}(A,y^{1-2s})$ as the $L^2$-based weighted Sobolev space give by 
	\begin{equation}
		L^2(A,y^{1-2s}):=\left\{ \wt u=\wt u(x,y): \R^{n+1}_+\to \R: \, 	\|\wt u\|_{L^{2}(D,y^{1-2s})}<\infty\right\},
	\end{equation}
	where
	\[
	\|\wt u\|_{L^{2}(A,y^{1-2s})}:=\bigg(\int_{A}y^{1-2s}\abs{\wt u}^{2}dxdy\bigg)^{1/2}.
	\]
	Define 
	\[
	H^{1}(A,y^{1-2s}):=\left\{\wt u\in L^{2}(A,y^{1-2s}):\,\nabla_{x,y}\wt u\in L^{2}(A,y^{1-2s})\right\},
	\]
	with $\nabla_{x,y}=\LC \nabla_x,\p_y\RC =\LC \nabla ,\p_y \RC$ being the total derivative for $(x,y)\in \mathbb{R}^{n+1}$. It is known that the $H^1(A,y^{1-2s})$ has a natural inner product structure that 
	\begin{equation}
		\LC \wt u , \wt v  \RC_{L^2(A,y^{1-2s})}:=\int_{A} y^{1-2s} \wt u \wt v \, dxdy, 
	\end{equation}
	for any $A\subseteq \R^{n+1}_+$, so that $\norm{\wt u}_{L^2(A,y^{1-2s})}^2 = \LC \wt u , \wt u \RC_{L^2(A,y^{1-2s})}$. In particular, let us denote another weighted Sobolev space by 
	\begin{equation}
		H^1_x(A, y^{1-2s}):= \left\{ \wt u\in L^2(A,y^{1-2s}) : \, \nabla \wt u \in L^2(A,y^{1-2s}) \right\}, 
	\end{equation}
	and clearly, $H^1_x(A,y^{1-2s})\subset H^1(A,y^{1-2s})$.

	\subsection{Nonlocal operators}
	We next review the nonlocal elliptic operator $\mathcal{L}^s$ ($0<s<1$), where  
	$$
	\mathcal{L}:=-\nabla \cdot \LC \sigma \nabla \RC
	$$ 
	is a second-order uniformly elliptic operator of divergence form. Note that when we define the nonlocal elliptic operator $\mathcal{L}^s$, we do not need to assume the condition $\sigma=1$ in $\Omega_e$. It is known that the nonlocal operator $\mathcal{L}^s= (-\nabla \cdot \sigma \nabla)^s$ can be defined by
	\begin{equation}
		\mathcal{L}^{s}:=\frac{1}{\Gamma(-s)}\int_{0}^{\infty}\left(e^{-t\mathcal{L}}-\mathrm{Id}\right)\,\frac{dt}{\tau^{1+s}},\label{eq:1111}
	\end{equation}
	where $e^{-t\mathcal{L}}$ stands for the heat kernel of $\p_t +\mathcal{L}$ in $\R^n\times (0,\infty)$, and $\mathrm{Id}$ denotes the identity map. Moreover, the operator $\mathcal{L}^s$ can be equivalently characterized by the famous Caffarelli-Silvestre \cite{CS_extension} and Stinga-Torrea extension problems \cite{ST10}.

	We can also define the bilinear form of the exterior problem \eqref{equ: main}.  It is known that the exterior problem \eqref{equ: main} is well-posed, and the DN map can be defined by the bilinear form 
	\begin{equation}\label{bilinear}
		\begin{split}
			\left\langle \mathcal{L}^s u, w \right\rangle  = \frac{1}{2}\int_{\R^n\times \R^n} \LC u(x)-u(z)\RC \LC w(x)-w(z) \RC K_s(x,z)\, dxdz,
		\end{split}
	\end{equation}
	where $K_s(x,z)$ is can be derived by the heat kernel 
	\begin{equation}\label{nonlocal kernel}
		K_s(x,z):=\frac{1}{\Gamma(-s)} \int_0^\infty p_t(x,z)\frac{dt}{t^{1+s}}.
	\end{equation}
	As $p_t(x,z)$ stands for the symmetric heat kernel for $\mathcal{L}$ so that 
	\begin{equation}
		\LC e^{-t\mathcal{L}} f\RC (x)=\int_{\R^n} p_t(x,z)f(z)\, dz, \text{ for }x\in \R^n, \ t>0,
	\end{equation}
	and 
	\begin{equation}
		\begin{cases}
			\LC \p _t+\mathcal{L}\RC  \LC e^{-t\mathcal{L}} f \RC   =0 &\text{ for }(x,t)\in \R^{n} \times (0,\infty), \\
			\LC e^{-t\mathcal{L}} f\RC (x,0)=f(x) &\text{ for }x\in \R^n.
		\end{cases}
	\end{equation}
	The heat kernel $p_t(x,z)$ enjoys pointwise estimates (cf. \cite{Davies90})
	\begin{equation}\label{heat kernel estimate}
		c_1 e^{-\alpha_1\frac{|x-z|^2}{t}}t^{-\frac{n}{2}}\leq p_t (x,z)\leq c_2 e^{-\alpha_2\frac{|x-z|^2}{t}}t^{-\frac{n}{2}}, \text{ for }x,z\in \R^n,
	\end{equation}
	for some positive constants $c_1,c_2,\alpha_1$ and $\alpha_2$. Therefore, we can obtain pointwise estimate for $K_s(x,z)$ from the formula \eqref{nonlocal kernel}
	\begin{equation}
		\frac{C_1}{\abs{x-z}^{n+2s}}\leq K_s(x,z)\leq \frac{C_2}{\abs{x-z}^{n+2s}}, \text{ for }x,z\in \R^n,
	\end{equation}
	for some constants $C_1,C_2>0$, which ensures the well-posedness of \eqref{equ: main} (see \cite[Section 3]{GLX} for detailed arguments). 
	This is equivalent to say that given any $f\in H^s(\Omega_e)$, the equation \eqref{equ: main} admits a unique solution $u\in H^s(\R^n)$. This implies the DN map $\Lambda_\sigma$ is well-defined, and there holds 
	\begin{equation}
		\begin{split}
			\left\langle  \Lambda_\sigma f , g \right\rangle_{H^{-s}(W)\times \wt H^s(W)} = \frac{1}{2}\int_{\R^n\times \R^n} \LC u(x)-u(z)\RC \LC w(x)-w(z) \RC K_s(x,z)\, dxdz
		\end{split}
	\end{equation}
	for any $f,g\in \wt H^s(W)$, which justifies \eqref{DN map}.

	\subsection{The extension problem}

	Recalling that the extension problem for the nonlocal operator $\mathcal{L}^s$ with $0<s<1$ is 
	\begin{equation}\label{equ: extension problem}
		\begin{cases}
			\nabla_{x,y}  \cdot (  y^{1-2s}\wt \sigma  \nabla_{x,y} \wt u  )  =0 &\text{ in }\R^{n+1}_+,\\
			\wt u(x,0)=u(x) &\text{ on }\p \R^{n+1}_+ =\R^n,
		\end{cases}
	\end{equation}
	where $\nabla_{x,y}=(\nabla , \p_y)$ with $\nabla =\nabla_x$.
	It is known that there holds the following relation holds  
	\begin{equation}\label{extension Neumann}
		-\lim_{y\to 0}y^{1-2s}\p_y \wt u= d_s ( -\nabla \cdot \sigma\nabla) ^s u\text{ in }\R^n,
	\end{equation}
	where $\R^{n+1}_+:= \left\{(x,y) \in \R^{n+1} : \, x\in   \R^n, \  y>0\right\}$, 
	\begin{equation}\label{d_s constant}
		d_s=\frac{\Gamma(1-s)}{2^{2s-1}\Gamma(s)} >0
	\end{equation}
	is a constant depending only on $s\in (0,1)$, and $\wt \sigma$ is an $(n+1)\times (n+1)$ matrix of the form
	\begin{align}\label{tilde sigma(x)}
		\wt \sigma(x)=\left( \begin{matrix}
			\sigma(x) \mathrm{I}_n& 0\\
			0 & 1 \end{matrix} \right).
	\end{align}
	with the $n\times n$ identity matrix $\mathrm{I}_n$. This type of extension problem was called the Caffarelli-Silvestre or Stinga-Torrea extension problem in the literature.

	\section{Monotonicity relation}\label{sec: mono}
	
	In this section, we want to show that $\sigma_1\geq \sigma_2$ a.e. in $\Omega$ implies that $\Lambda_{\sigma_1}\geq \Lambda_{\sigma_2}$ in the quadratic sense. More specifically, with the discussions from the previous section, we use the notion
	\begin{equation}\label{sense quadratic}
		\Lambda_{\sigma_1} \geq \Lambda_{\sigma_2} \iff 	\left\langle  \Lambda_{\sigma_1} f , f \right\rangle_{H^{-s}(W)\times \wt H^s(W)} \geq 	\left\langle  \Lambda_{\sigma_2} f , f \right\rangle_{H^{-s}(W)\times \wt H^s(W)} 
	\end{equation}
	for any $f\in C^\infty_c(W)$. In what follows, we may use $\langle \cdot, \cdot \rangle \equiv \langle \cdot ,\cdot \rangle_{H^{-s}(W)\times \wt H^s(W)}$ to denote the duality pairing, which simplifies the notations, provided that there are no further confusions.

	\begin{lemma}[Monotonicity relations]\label{Lemma: monotonicity}
		Let $\Omega  \subset \R^n$ and $W\Subset \Omega_e$ be bounded open sets with Lipschitz boundaries, for $n\in \N$. Let $\sigma_j$ be a bounded positive coefficient satisfying \eqref{ellipticity}, and  $\Lambda_{\sigma_j}$ be the DN map of \eqref{equ: main j=1,2}, for $j=1,2$. Then there hold that 
		\begin{equation}\label{mono lemma}
			\begin{split}
				&\quad \, 	\int_{\Omega \times (0,\infty)} y^{1-2s}\LC \sigma_1-\sigma_2\RC \big| \nabla  \wt u_1^f \big|^2 \, dxdy \\
				& \leq d_s \left\langle (\Lambda_{\sigma_1}-\Lambda_{\sigma_2}) f , f \right\rangle \\
				& \leq 	\int_{\Omega \times (0,\infty)} y^{1-2s}\LC \sigma_1-\sigma_2\RC \big| \nabla \wt u_2^f \big|^2 \, dxdy,
			\end{split}
		\end{equation}
		and 
		\begin{equation}\label{mono lemma 0}
			\begin{split}
				&\quad \,	\int_{\Omega \times (0,\infty)} y^{1-2s}\frac{\sigma_2}{\sigma_1}\LC \sigma_1-\sigma_2\RC \big| \nabla \wt u_2^f \big|^2 \, dxdy\\
				& \leq d_s \left\langle (\Lambda_{\sigma_1}-\Lambda_{\sigma_2}) f , f \right\rangle \\
				& \leq 	\int_{\Omega \times (0,\infty)} y^{1-2s}\LC \sigma_1-\sigma_2\RC \big| \nabla \wt u_2^f \big|^2 \, dxdy,
			\end{split}
		\end{equation}
		where $\wt u_j^f$ is the solution to the extension problem \eqref{equ: extension problem} as $\sigma=\sigma_j$ and $\wt u_j^f(x,0)=u_j^f(x)$ in $\R^n$, where $u_j^f \in H^s(\R^n)$ is the solution to \eqref{equ: main j=1,2} as $j=2$, with the constant $d_s>0$ in \eqref{CS extension relation}.			
	\end{lemma}

	\begin{remark}
		Let us point out that 
		\begin{enumerate}[(i)]
			\item  As $s=1$, the monotonicity implication from the conductivity to the DN map is easy to derive using the bilinear form and integration by parts. However, as shown in Section \ref{sec: prel}, it is known that the bilinear form \eqref{bilinear} may not be useful to derive this monotonicity relation. Instead, we will use the Caffarelli-Silvestre type extension problem \eqref{equ: extension problem} to prove the above lemma.
			
			\item Note that the both sides of \eqref{mono lemma} contain solutions $\wt u_1^f$ and $\wt u_2^f$, but in the both sides of \eqref{mono lemma 0} only consist the solution $\wt u_2^f$. The relation \eqref{mono lemma} is already useful in the determination of the leading coefficients. For further applications in inverse problems (such as inverse obstacle problems), the monotonicity formula \eqref{mono lemma 0} would be needed. However, we do not pursue this problem in this article.
		\end{enumerate}
		
	\end{remark}
	
	\begin{proof}[Proof of Lemma \ref{Lemma: monotonicity}]
		Let $\wt \sigma_j$ be of the form \eqref{tilde sigma(x)} as $\sigma=\sigma_j$, and $\wt u_j^f$ denote the solution to \eqref{equ: extension problem} as $\wt u =\wt u^f _j$ with $\wt u_j^f(x,0)=u_j^f(x)$ in $\R^n$, for $j=1,2$. Here $u_j^f\in H^s(\R^n)$ is the solution to \eqref{equ: main j=1,2} for $j=1,2$. Then there holds that 
		\begin{equation}\label{CS extension relation}
			-\lim_{y\to 0}y^{1-2s}\p_y \wt u_j^f= d_s (-\nabla \cdot \sigma_j \nabla )^s u_j^f \text{ in }\R^n,
		\end{equation}
		for $j=1,2$, where $d_s$ is a positive constant given by \eqref{d_s constant}. On the one hand, there holds 
		\begin{equation}\label{mono lemma 1}
			\begin{split}
				\left\langle \Lambda_{\sigma_j} f, f \right\rangle &= \int_{W}(-\nabla \cdot \sigma_j\nabla )^s u_j^f \cdot f \, dx ,
			\end{split}
		\end{equation}
		for $j=1,2$. On the other hand, we have 
		\begin{equation}
			\begin{split}
				0&=\int_{\R^{n+1}_+} \nabla_{x,y}\cdot \big( y^{1-2s}\wt \sigma \nabla_{x,y} \wt u^f_1 \big) \wt u_1^f\, dxdy \\
				&= -\int_{\R^n}\big( \lim_{y\to 0}y^{1-2s} \p_y \wt u^f_1 \big) \wt u_1^f (x,0)\, dx - \int_{\R^{n+1}_+}y^{1-2s}\wt \sigma_j \nabla_{x,y}\wt u^f_1\cdot \nabla_{x,y}\wt u^f_1 \, dxdy \\
				&= \underbrace{d_s \int_{\R^n} (-\nabla \cdot \sigma_j\nabla)^s u_1^f \cdot u_1^f \, dx}_{\text{By \eqref{extension Neumann}}} - B_{\wt \sigma_1}(\wt u_1^f, \wt u_1^f) \\
				&=\underbrace{d_s\int_{W}(-\nabla \cdot \sigma_j\nabla)^s u_1^f \cdot f \, dx}_{\text{Since }u^f \text{ solves \eqref{equ: main j=1,2} and }\left. u_1^f\right|_{\Omega_e}=f} - B_{\wt \sigma_1}(\wt u_1^f, \wt u_1^f),
			\end{split}
		\end{equation}
		which implies that 
		\begin{equation}\label{mono lemma 2}
			\begin{split}
				B_{\wt \sigma_1}(\wt u_1^f, \wt u_1^f) =d_s \left\langle \Lambda_{\sigma_1}f, f \right\rangle,
			\end{split}
		\end{equation}
		where 
		\begin{equation}
			\begin{split}
				B_{\wt \sigma_j}(\wt u, \wt w)&:=   \int_{\R^{n+1}_+} y^{1-2s}\wt \sigma_j \nabla_{x,y}\wt u \cdot \nabla_{x,y}\wt w \, dxdy\\
				&\ = \int_{\R^{n+1}_+}y^{1-2s}\sigma_j \nabla \wt u\cdot \nabla \wt w \, dxdy + \int_{\R^{n+1}_+}y^{1-2s}\p_y \wt u \p_y \wt w \, dxdy
			\end{split}
		\end{equation}
		is a symmetric bilinear form, for $j=1,2$ and any functions $\wt u$ and $\wt w$. 
		
		Similarly, we also have 
		\begin{equation}\label{mono lemma 3}
			B_{\wt \sigma_1}(\wt u_1^f, \wt u_2^f) =d_s \left\langle \Lambda_{\sigma_1}f, f \right\rangle,
		\end{equation}
		
		Combining \eqref{mono lemma 1}, \eqref{mono lemma 2} and \eqref{mono lemma 3}, we have 
		\begin{equation}\label{mono lemma 4}
			\begin{split}
				d_s\left\langle \Lambda_{\sigma_1}f, f \right\rangle&=B_{\wt \sigma_1}(\wt u_1^f, \wt u_1^f)  = 	B_{\wt \sigma_1}(\wt u_1^f, \wt u_2^f) , \\
				d_s\left\langle \Lambda_{\sigma_2}f, f \right\rangle&=B_{\wt \sigma_2}(\wt u_2^f, \wt u_2^f) .
			\end{split}
		\end{equation}
		Direct computations yield that 
		\begin{equation}\label{mono lemma 5}
			\begin{split}
				0& \leq B_{\wt \sigma_1} (\wt u_1^f -\wt u_2^f, \wt u_1^f -\wt u_2^f )
				\\
				&= B_{\wt \sigma_1}(\wt u_1^f , \wt u_1^f ) -2 B_{\wt \sigma_1} (\wt u_1^f ,\wt u_2^f ) + B_{\wt \sigma_1} (\wt u_2^f ,\wt u_2^f )\\
				&=	-d_s\left\langle \Lambda_{\sigma_1}f, f \right\rangle + d_s\left\langle \Lambda_{\sigma_2}f, f \right\rangle +B_{\wt \sigma_1}(\wt u_2^f ,\wt u_2 ^f )-B_{\wt \sigma_2}(\wt u_2^f, \wt u_2^f),
			\end{split} 
		\end{equation}
		and this implies that 
		\begin{equation}\label{mono lemma 6}
			\begin{split}
				&\quad \, d_s \LC \left\langle \Lambda_{\sigma_1}f, f \right\rangle - \left\langle \Lambda_{\sigma_2}f, f \right\rangle \RC \\
				&\leq  \bigg( \int_{\R^{n+1}_+}y^{1-2s}\sigma_1 \big| \nabla \wt u_2^f \big|^2 \, dxdy +  \int_{\R^{n+1}_+}y^{1-2s}|\p_y \wt u_2^f |^2 \, dxdy \bigg) \\ 
				&\quad \, - \bigg( \int_{\R^{n+1}_+}y^{1-2s}\sigma_2 \big| \nabla \wt u_2^f \big|^2 \, dxdy +  \int_{\R^{n+1}_+}y^{1-2s}|\p_y \wt u_2^f |^2 \, dxdy \bigg) \\
				&\leq \int_{\R^{n+1}_+}y^{1-2s} \LC \sigma_1 -\sigma_2 \RC \big|\nabla \wt u_2^f \big|^2 \, dxdy\\
				&= \int_{\Omega \times (0,\infty)}y^{1-2s} \LC \sigma_1 -\sigma_2 \RC \big|\nabla \wt u_2^f \big|^2 \, dxdy,
			\end{split}
		\end{equation}
		where we used $\sigma_1=\sigma_2$ in $\Omega_e$. This proves the right-hand side ordering in \eqref{mono lemma}.
		Interchanging the indices $j=1,2$ in \eqref{mono lemma 6}, we can obtain the left-hand side in \eqref{mono lemma}. 
		
		Finally, for the left-hand side in \eqref{mono lemma 0}, let us interchange $\sigma_1$ and $\sigma_2$ in \eqref{mono lemma 5} and \eqref{mono lemma 6}, then we have 
		\begin{equation}
			\begin{split}
				&\quad \, d_s \left\langle (\Lambda_{\sigma_1}-\Lambda_{\sigma_2}) f, f\right\rangle  \\
				&= \int_{\Omega\times (0,\infty)} y^{1-2s} \LC\sigma_1- \sigma_2\RC \big| \nabla \wt u_1^f|^2 \, dxdy \\
				&\quad \, +\int_{\Omega\times (0,\infty)} y^{1-2s} \sigma_2 \big| \nabla \big( \wt u_2^f-\wt u_1^f \big)\big|^2 \, dxdy \\
				&= \int_{\Omega\times (0,\infty)} y^{1-2s} \big( \sigma_1 \big| \nabla \wt u_1^f\big|^2 + \sigma_2 \big| \nabla \wt u_2^f \big|^2 -2\sigma_2 \nabla \wt u_1^f \cdot \nabla \wt u_2^f  \big)\, dxdy \\
				&=\underbrace{\int_{\Omega \times (0,\infty)} y^{1-2s}\sigma_1 \left| \nabla \wt u_1^f -\frac{\sigma_2}{\sigma_1}\nabla \wt u_2^f \right|^2 \, dxdy}_{\text{nonnegative}} \\
				&\quad \, + \int_{\Omega\times (0,\infty)} y^{1-2s} \LC \sigma_2 -\frac{\sigma^2}{\sigma_1}\RC \big| \nabla \wt u^f_2 \big|^2 \, dxdy \\
				&\geq \int_{\Omega\times (0,\infty)} y^{1-2s} \frac{\sigma_2}{\sigma_1}\LC \sigma_1 -\sigma_2\RC \big| \nabla \wt u^f_2 \big|^2 \, dxdy,
			\end{split}
		\end{equation}
		which proves \eqref{mono lemma 0}.
	\end{proof}

	\begin{corollary}
		Adopting all assumptions in Lemma \ref{Lemma: monotonicity}, let $\sigma_1,\sigma_2$ be two positive bounded coefficients, then 
		\[
		\sigma_1 \geq \sigma_2 \text{ implies }\Lambda_{\sigma_1}\geq \Lambda_{\sigma_2}.
		\] 
	\end{corollary}
	
	\begin{proof}
		The proof can be seen using either \eqref{mono lemma} or \eqref{mono lemma 0}. 
	\end{proof}
	
	\begin{remark}
		From the above discussions, one can see the monotonicity relations depend only on the gradient of certain solutions with respect to the transversal direction (i.e., $x$-variable), but not $y\in (0,\infty)$.
	\end{remark}

	\section{Localized potentials and Runge approximation}\label{sec: localized potential}

	We first rewrite the extension problem \eqref{equ: extension problem}. 
	If $u\in H^s(\R^n)$ is a solution to \eqref{equ: main} into
	\begin{equation}\label{equ: extension problem 2}
		\begin{cases}
			\nabla_{x,y}  \cdot (  y^{1-2s}\wt \sigma  \nabla_{x,y} \wt u  )  =0 &\text{ in }\R^{n+1}_+,\\
			\displaystyle\lim_{y\to 0}y^{1-2s}\p_y \wt u =0 &\text{ on }\Omega \times \{0\}, \\
			\wt u(x,0)=f(x) &\text{ on }\Omega_e\times \{0\}.
		\end{cases}
	\end{equation}
	Thanks to the monotonicity formulas given in the previous section, both sides in \eqref{mono lemma} and \eqref{mono lemma 0} depend only on the gradient along the $x$-direction but are independent of $y>0$. Thus, we can construct localized potentials for the extension problem \eqref{equ: extension problem 2} by showing the Runge approximation with respect to the $x$-variable. More specifically, we have the next result.

	\begin{theorem}[Runge approximation]\label{thm: density}
		Let $\Omega\subset \R^n$ be a bounded Lipschitz domain for $n\geq 3$, and $B, D\subseteq \overline{\Omega}$ be measurable sets, $B\setminus \overline{D}$ possess positive measure, and $\overline{\Omega} \setminus \overline{D}$ is connected to $\p \Omega$. Then there exist  functions $v =v(x)\in H^1(B\cup D)$, and $0\not\equiv\beta_1=\beta_1(y)\in C^\infty_c((0,\infty))$, such that the function $v(x)\beta_1(y)$ can be approximated in the $H^1_x(B\cup D,y^{1-2s})$-norm by solutions $\wt u$ of \eqref{equ: extension problem 2}, and satisfies 
		\begin{equation}\label{property of function v}
			\nabla v| _{D}\equiv 0  \quad \text{and} \quad \nabla v|_{B}\not\equiv 0.
		\end{equation}
	\end{theorem}
	
	We will explicitly construct the functions $v(x)$ and $\beta_1(y)$ in the proof of Theorem \ref{thm: density}. Assuming Theorem \ref{thm: density} is true, we can have the following existence of localized potentials.
	
	\begin{corollary}[Localized potentials]\label{Cor: localized potentials}
		Let $B,D\subseteq \overline{\Omega}$ be nonempty measurable sets, $B\setminus \overline{D}$ possess positive measure, and $\overline{\Omega} \setminus  \overline{D}$ is connected to $\p \Omega$. Let $W\Subset \Omega_e$ be a nonempty open subset, then there exists a sequence $\{f_k\}_{k=1}^\infty\subset C^\infty_c(W)$ such that 
		\begin{equation}
			\begin{split}
				\int_{B\times (0,\infty)} y^{1-2s} \left| \nabla \wt u^{f_k}\right|^2 \, dxdy &\to \infty, \\
				\int_{D\times (0,\infty)} y^{1-2s}  \left| \nabla \wt u^{f_k}\right|^2 \, dxdy &\to 0,
			\end{split}
		\end{equation}
		as $k\to \infty$, where $u^{f_k}\in H^1(\R^{n+1}_+,y^{1-2s})$ is the solution to \eqref{equ: extension problem 2} with 
		$$
		\begin{cases}
			\displaystyle\lim_{y\to 0}y^{1-2s}\p_y \wt u^{f_k}=0 &\text{ in }\Omega ,\\
			\wt u^{f_k}(x,0)=f_k(x) & \text{ in }\Omega_e,
		\end{cases}
		$$
		for all $k\in \N$.
	\end{corollary}

	\begin{proof}
		By using Theorem \ref{thm: density}, there exist functions $v\in H^1(B\cup \overline{D})$, $0\neq \beta_1 \in C^\infty_c((0,\infty))$, and sequence of solutions $\wt v_{\wt f_k}\in H^1(\R^{n+1}_+, y^{1-2s})$ of 
		\begin{equation}\label{equ: extension problem 3}
			\begin{cases}
				\nabla_{x,y}  \cdot (  y^{1-2s}\wt \sigma  \nabla_{x,y} \wt v^{\wt f_k}  )  =0 &\text{ in }\R^{n+1}_+,\\
				\displaystyle\lim_{y\to 0}y^{1-2s}\p_y  \wt v^{\wt f_k}=0 &\text{ on }\Omega \times \{0\}, \\
				\wt v^{\wt f_k}(x,0)=\wt f_k(x) &\text{ on }\Omega_e\times \{0\},
			\end{cases}
		\end{equation}
		for all $k\in \N$, such that
		\begin{equation}
			\int_{B\times (0,\infty)}y^{1-2s}\big| \nabla \wt v^{\wt f_k}\big|^2 \, dxdy \rightarrow  \int_{B\times (0,\infty)}y^{1-2s} \beta_1^2\left| \nabla v\right|^2 \, dxdy > 0,
		\end{equation}
		and 
		\begin{equation}
			\int_{D\times (0,\infty)}y^{1-2s}\big| \nabla \wt v^{\wt f_k}\big|^2 \, dxdy \rightarrow  \int_{D\times (0,\infty)}y^{1-2s}\beta_1^2\left| \nabla v\right|^2 \, dxdy = 0.
		\end{equation}
		Notice that the value $\int_{D\times (0,\infty)}y^{1-2s}\big| \nabla \wt v^{\wt f_k}\big|^2 \, dxdy>0$ for all $k\in \N$ thanks to the UCP of \eqref{equ: extension problem 2}.
		Let 
		$$
		f_k := \frac{\wt f_k}{\sqrt{\int_{D\times (0,\infty)}y^{1-2s}\big| \nabla \wt v^{\wt f_k}\big|^2 \, dxdy}},
		$$ 
		then $\wt u^{f_k}=\frac{\wt v^{\wt f_k}}{\sqrt{\int_{D\times (0,\infty)}y^{1-2s}\big| \nabla \wt v^{\wt f_k}\big|^2 \, dxdy}}$ solves \eqref{equ: extension problem 2} with the boundary data $f=f_k\in C^\infty_c(W)$, for $k\in \N$. This proves the assertion.		
	\end{proof}

	Now, it remains to prove Theorem \ref{thm: density}.

	\begin{proof}[A formal proof of Theorem \ref{thm: density}]
		Let us consider the case $s=1/2$ and $\sigma=\mathrm{I}_n$ (the $n\times n$ identity matrix). For simplicity, let us define $C:=B\cup D \subseteq \overline{\Omega}$, such that $\overline{\Omega}\setminus \overline{C}$ is connected to $\p \Omega$. Without loss of generality, we may assume that $\p C$ is Lipschitz. Given $\psi \in \wt H^{-1}(C)$, by the duality argument of the Hahn-Banach theorem, we only need to claim 
		\begin{equation}\label{duality 0}
			\int_0^\infty \left\langle \wt u^f(\cdot, y) , \psi (\cdot)\right\rangle_C \, dy=0
		\end{equation}
		implies 
		\begin{equation}\label{duality 0-1}
			\int_0^\infty \big\langle  \wt \beta_1(y) v(\cdot) , \psi(\cdot) \big\rangle_C \, dy=0 ,
		\end{equation}
		where $\langle \cdot, \cdot\rangle _C$ denotes the duality pairing between $H^1(C)$ and its dual space $\wt H^{-1}(C)$\footnote{Note that the \eqref{duality 0} is equivalent to $\int_C \big(\int_0^\infty \wt u^f(x, y) \, dy \big) \psi(x) \, dx $ and \eqref{duality 0-1} is equivalent to $\int_C \big( \int_0^\infty \wt \beta_1(y)v(x) \, dy \big) \psi(x) \, dx$.}. 
		Here the function $v\in H^1(C)$ will satisfy \eqref{property of function v}, and $\wt \beta_1 \in C^\infty_c((0,\infty))$ fulfills $\int_0^\infty \wt \beta_1(y)\, dy=1$, $\wt\beta_1 \geq 0$, and  $\supp\big( \wt \beta_1 \big) \subset (1,2)$.
		Moreover, let us define $\wt\beta_k(y)=k^{-1}\wt\beta_1(y/k)$ for all $k\in \N$ and $y>0$, then there holds $\int_0^\infty \wt \beta_k(y) \, dy =1$ as well.
		
		To conclude the Hahn-Banach argument, let us consider the adjoint problem 
		\begin{equation}\label{adjoint problem formal}
			\begin{cases}
				\Delta_{x,y} w =\psi &\text{ in }\R^{n+1}_+, \\
				\displaystyle \lim_{y\to 0}\p_y w =0 & \text{ in }\Omega \times \{0\}, \\
				w=0 &\text{ in }\Omega_e \times \{0\}.			
			\end{cases}
		\end{equation}
		Since $C\subseteq\overline{\Omega}$ is a measurable set with $C\cap \p \Omega\neq \emptyset$, note that  
		\begin{equation}
			\bigg| \int_0^\infty \left\langle \wt u_f , \psi \right\rangle \, dy \bigg| \leq \norm{\psi}_{\wt H^{-1}(C)}\int_0^\infty \left\| \wt u_f (\cdot, y \right\|_{H^1(C)}\, dy<\infty ,
		\end{equation}
		then the assumption in \eqref{duality 0} can be rewritten as 
		\begin{equation}
			\begin{split}
				0&=	\int_0^\infty \left\langle \wt u^f , \psi \right\rangle_{C} \, dy\\
				&= \int_0^\infty \int_{\R^n} \LC \Delta_{x,y} w \RC  \wt u^f\, dxdy \\
				& = \underbrace{\int_{\R^n} w(x,0) \big( \lim_{y\to 0}\p_y \wt u^f) \, dx}_{:=(I)} + \underbrace{ \int_0^\infty \int_{\R^n} w \LC \Delta_{x,y}\wt u^f \RC dxdy}_{=0 \text{ since $\wt u^f$ solves \eqref{equ: extension problem 2}}}\\
				&\quad \, -\int_{\R^n} \big( \lim_{y\to 0}\p_y w \big) \wt u^f(x,0) \, dx  \\
				&= -\int_{W} \big( \lim_{y\to 0}\p_y w \big) f \, dx ,
			\end{split}
		\end{equation}
		where we used both boundary information of $\wt u^f$ and $w$ from the equations \eqref{equ: extension problem 2} and \eqref{adjoint problem formal} such that the term $(I)=0$. Since $f\in C^\infty_c(W)$ is arbitrary, there must hold that $\lim_{y\to 0}\p_y w=0$ in $W$. Thanks to the (weak) UCP and $\overline{\Omega}\setminus \overline{C}$ is connected to $\p \Omega$, then it follows that 
		$$
		w\equiv 0 \text{ in }\LC \Omega_e \cup (\Omega \setminus \overline{C}) \RC \times (0,\infty).
		$$		
		Particularly, there holds that
		\begin{equation}\label{boundary data w =0}
			w|_{\LC \p \Omega \cup \p C\RC \times (0,\infty) }=\left. \p_\nu w\right|_{\LC \p \Omega \cup \p C\RC \times (0,\infty) }=0,
		\end{equation}
		and $\lim_{y\to 0}\p_y w =0$ in $\R^n$.
		
		We can conclude the proof by taking the function $v\in H^1(C)$, which can be extended to an $H^1(\R^n)$ function and satisfies 
		\begin{equation}\label{equation v formal}
			\Delta v =0 \text{ in }C.
		\end{equation} 
		Similar to \cite[Section 3]{CGRU2023reduction}, we have  
		\begin{equation}\label{HB comp formal}
			\begin{split}
				-\langle \psi, v\rangle_C &= - \bigg\langle \psi, \int_0^\infty \wt\beta_k(y)v\, dy   \bigg\rangle_C\\
				&= -\lim_{k\to \infty} \bigg\langle \psi, \int_0^\infty \wt\beta_k(y)v\, dy   \bigg\rangle_C \\
				&=\underbrace{- \lim_{k\to \infty}\int_{\Omega \times (0,\infty)} \wt \beta_k w \Delta v\, dxdy}_{:=(II)}+ \lim_{k\to \infty}\int_{\Omega \times (0,\infty)} v \p_y \wt \beta_k \p_y w\, dxdy\\
				&= \lim_{k\to \infty} k^{-2}\int_{\Omega \times (k,2k)} v \p_y  \wt\beta_1 \p_y w \, dxdy\\
				&=0,
			\end{split}
		\end{equation}
		where we also used \eqref{boundary data w =0} to get rid of boundary contributions. The term $(II)=0$ thanks to the facts $\Delta v =0$ in $C$ and $w=0$ in $\LC \Omega \setminus \overline{C}\RC \times (0,\infty)$.

		Finally, since $v\in H^1(C)$ is a solution to \eqref{equation v formal} with $C=B\cup D$ and $\overline{B}\cap \overline{D}=\emptyset$, one can simply take $v$ to be piecewise defined solution, which is a nonconstant solution in $B$, and $v=1$ in $D$. Therefore, with these special choices at hand, we can imply that \eqref{property of function v} holds. This demonstrates the ideas of the proof of the density, which is very similar to the case shown in \cite[Section 3]{CGRU2023reduction}.		
	\end{proof}

	\begin{remark}
		From the first identity in \eqref{HB comp formal}, we can see that 
		$$
		\bigg\langle \psi, \int_0^\infty \wt\beta_k(y)v\, dy   \bigg\rangle_C=0, \text{ for all } k\in \N,
		$$ 
		which, of course, holds as $k=1$. The rigorous proof will follow the same idea and approach as in the formal one.
	\end{remark}

	The rigorous proof can be concluded by introducing suitable cutoff functions in both $x$ and $y$ directions, given as in \cite[Section 3]{CGRU2023reduction}.

	\begin{proof}[Proof of Theorem \ref{thm: density}]
		The argument is similar to the rigorous proof \cite[Proposition 3.1]{CGRU2023reduction}.
		As in the formal proof, let $C=B\cup D$ with $C\cap \p \Omega \neq \emptyset$, and $\psi \in \wt H^{-1}(C)$. Then we have 
		\begin{equation}
			\bigg| \int_0^\infty y^{1-2s}\left\langle \wt u^f (\cdot, y),\psi (\cdot) \right\rangle_C \, dy \bigg|\leq \norm{\psi}_{\wt H^{-1}(C)} \norm{\wt u_f}_{H^1(C,y^{1-2s})}<\infty,
		\end{equation}
		and by the duality argument of the Hahn-Banach theorem, it suffices to claim that 
		\begin{equation}\label{duality 1}
			\int_0^\infty y^{1-2s}\left\langle \wt u^f (\cdot, y),\psi (\cdot) \right\rangle_C \, dy =0 \quad 
		\end{equation}
		implies 
		\begin{equation}\label{duality 2}
			\int_0^\infty y^{1-2s}\left\langle \beta_1(y) v(\cdot),\psi(\cdot)\right\rangle_C \, dy=0,
		\end{equation}
		where $\langle \cdot, \cdot\rangle _C$ denotes the duality pairing between $H^1(C)$ and its dual space $\wt H^{-1}(C)$\footnote{Note that the \eqref{duality 1} is equivalent to $\int_C \big(\int_0^\infty y^{1-2s}\wt u^f(x, y) \, dy \big) \psi(x) \, dx $ and \eqref{duality 2} is equivalent to $\int_C \big( \int_0^\infty y^{1-2s} \beta_1(y)v(x) \, dy \big) \psi(x) \, dx$.}. As we shall see, the function $v\in H^1(C)$ could be an arbitrary function satisfying 
		\begin{equation}\label{equation v}
			\nabla \cdot (\sigma\nabla v)=0  \text{ in }C.
		\end{equation}
		Moreover, $\beta_1$ is a suitable cutoff function with respect to the $y$-direction, which will be given in Appendix \ref{sec: appendix}.

		To this end, consider the adjoint problem 
		\begin{equation}\label{adjoint problem}
			\begin{cases}
				\nabla_{x,y}\cdot \LC y^{1-2s}\wt \sigma \nabla_{x,y} w \RC = y^{1-2s}\psi &\text{ in }\R^{n+1}_+, \\
				\displaystyle \lim_{y\to 0}y^{1-2s}\p_y w =0 & \text{ in }\Omega \times \{0\}, \\
				w=0 &\text{ in }\Omega_e \times \{0\}.			
			\end{cases}
		\end{equation}
		The solvability of the above problem \eqref{adjoint problem} can be found in \cite[Section 3]{CGRU2023reduction}, so we do not give further details about it. 
		The rest of the argument is completely the same as the proof of \cite[Proposition 3.1]{CGRU2023reduction}, with suitable cutoff function arguments in both $x$ and $y$ directions; by using \eqref{duality 1}, one can obtain 
		\begin{equation}
			\int_W f\lim_{y\to 0 } y^{1-2s}\p_y w \, dx=0, \quad \text{for any }f\in C^\infty_c(W).
		\end{equation}
		This implies that $\lim_{y\to 0}y^{1-2s}\p_y w =0$ in $W$. Combining with $w=0$ in $W\times \{0\}$ and $\overline{\Omega} \setminus \overline{C}$ is connected to $\p \Omega$, the UCP implies that 
		$$
		w=0 \text{ in } \LC \Omega_e \cup \LC \Omega \setminus \overline{C}\RC \RC \times (0,\infty).
		$$ 
		This particularly implies also implies that $w=\p_\nu w =0 $ on $\LC \p \Omega \cup \p C \RC \times (0,\infty)$, and $\lim_{y\to 0}y^{1-2s}\p_yw =0$ in $\R^n$.

		Finally, given any function $v\in H^1(C)$ that satisfies \eqref{equation v}, then the same arguments as in the proof of \cite[Proposition 3.1]{CGRU2023reduction} and \eqref{HB comp formal}, the identity \eqref{duality 2} must hold. In other words, for any solution $v\in H^1(C)=H^1(B\cup D)$ to \eqref{equation v}, the function $v(x)\beta_1(y)$ can be approximated by solutions $\wt u$ of \eqref{equ: extension problem} in $H^1(B\cup D, y^{1-2s})$, then this concludes the proof.		
	\end{proof}
	
	\begin{remark}
		In the proof of \cite[Proposition 3.1]{CGRU2023reduction}, there are more (smooth) cutoff functions involved. However, the proofs of Theorem \ref{thm: density} and \cite[Proposition 3.1]{CGRU2023reduction} are the same, and there are no major changes to the whole argument. We skip the details.
	\end{remark}

	\section{Proof of the local uniqueness}\label{sec: thm}
	
	We are now ready to prove Theorem \ref{thm: local}.
	
	\begin{proof}[Proof of Theorem \ref{thm: local}]
		We want to claim $\sigma_1 =\sigma_2$ in $\mathcal{O}$ by using the monotonicity relation and localized potentials. Since the condition \eqref{pointwise mono assumption} holds, we may assume $\sigma_1 -\sigma_2 \geq 0$ in $\mathcal{O}$, and there exists $B\subseteq \mathcal{O}$ such that 
		\begin{equation}\label{mono sigma1-sigma2}
			\sigma_1 -\sigma_2\geq \delta>0 \text{ in } B,
		\end{equation}
		for some $\delta>0$. Additionally, we can assume $D:=\Omega \setminus \overline{\mathcal{O}}$, such that $B \setminus \overline D$ possesses positive measure since $\overline{B}\cap \overline{D}=\emptyset$, and $\overline{\Omega}\setminus   \overline{D}$ is connected to $\p \Omega$.

		Next, the monotonicity relation \eqref{mono lemma 0} yields that 
		\begin{equation}\label{pf of local 1}
			\begin{split}
				0 &= d_s \left\langle (\Lambda_{\sigma_1}-\Lambda_{\sigma_2}) f_k , f_k \right\rangle  \\
				& \geq \int_{\Omega \times (0,\infty)} y^{1-2s}\frac{\sigma_2}{\sigma_1}\LC \sigma_1-\sigma_2\RC \big| \nabla \wt u_2^{f_k} \big|^2 \, dxdy \\
				&\geq -C \int_{D\times (0,\infty)} y^{1-2s} \big| \nabla \wt u_2^{f_k} \big|^2 \, dxdy +c\delta \int_{B\times (0,\infty)}y^{1-2s} \big| \nabla \wt u^{f_k}_2 \big|^2 \, dxdy,
			\end{split}
		\end{equation}
		for some constants $c,C>0$ independent of $\wt u_2^{f_k}$, where $\big\{ \wt u_2^{f_k}\big\}_{k\in \N} \subset H^1(\R^{n+1}_+, y^{1-2s})$ are localized potentials constructed by Corollary \ref{Cor: localized potentials}. Here 
		$C:= \big\| \frac{\sigma_2}{\sigma_1} \LC \sigma_1 -\sigma_2 \RC \big\|_{L^\infty(\Omega)}$ and $c:=\min_{\overline{\Omega}} \frac{\sigma_2}{\sigma_1}>0$. Hence, taking $k\to \infty$ in the inequality \eqref{pf of local 1}, we can obtain
		\begin{equation}
			\begin{split}
				0 \geq -C \underbrace{\int_{D\times (0,\infty)} y^{1-2s} \big| \nabla \wt u^{f_k}_2 \big|^2 \, dxdy}_{\to 0 \text{ as }k\to \infty} +c\delta \underbrace{\int_{B\times (0,\infty)}y^{1-2s} \big| \nabla \wt u^{f_k}_2 \big|^2 \, dxdy}_{\to \infty \text{ as }k\to \infty} \rightarrow +\infty,
			\end{split}
		\end{equation}
		as $k\to \infty$, which leads to a contradiction. Therefore, the inequality \eqref{mono sigma1-sigma2} cannot be true, which implies that $\sigma_1\leq \sigma_2$ in $\mathcal{O}$.

		Similarly, when $\sigma_1-\sigma_2\leq 0$ in $\mathcal{O}$, we may assume that $\sigma_1-\sigma_2 \leq -\delta'<0$ in $B\subsetneq \mathcal{O}$ for some $\delta'>0$, this will lead a contradiction too. More concretely, with the same notations as before, one can use the monotonicity relation \eqref{mono lemma 0} again, 
		\begin{equation}
			\begin{split}
				0  &\leq \int_{\Omega \times (0,\infty)} y^{1-2s}\LC \sigma_1-\sigma_2\RC \big| \nabla \wt u_2^{f_k} \big|^2 \, dxdy \\
				& \leq C \int_{D\times (0,\infty)}  y^{1-2s}\big| \nabla \wt u_2^{f_k} \big|^2 \, dxdy -\delta'  \int_{B\times (0,\infty)}y^{1-2s}\big| \nabla \wt u_2^{f_k}\big|^2 \, dxdy \\
				& \to -\infty,
			\end{split}
		\end{equation}
		which leads to a contradiction, too. Thus, $\sigma_1 =\sigma_2$ in $B$, concluding the proof.		
	\end{proof}

	\begin{remark}
		It would also be interesting to use the monotonicity approach to study the Lipschitz stability for piecewise analytic coefficients and inverse obstacle problems as we mentioned earlier.
	\end{remark}

	\appendix
	
	\section{The auxiliary function}\label{sec: appendix}
	
	The function introduced in this work, $\beta_1=\beta_1(y)$, was constructed in the proof of \cite[Proposition 3.1]{CGRU2023reduction}. For the sake of completeness, let us collect the existing construction of $\beta_k$ ($k\in \N$) as follows: Given $b\in (0,1)$, let $\gamma_b:(0,\infty)\to (0,b)$ be a smooth function with $\supp \LC \gamma_b \RC \subset [0,\frac{2-b}{1-b}]$ and $\gamma_b(y)=b$ for $y\in [1,\frac{1}{1-b}]$. One can also assume that 
	$$\int_0^\infty \gamma_b(y)\, dy=\frac{b}{1-b} \quad \text{and} \quad \left| \p_y^\ell \gamma_b(y)\right| \leq C,
	$$
	for $\ell=0,1,2$, where $C>0$ is a constant independent of $b\in (0,1)$. Let 
	\begin{equation}
		\begin{split}
			I_{b,k}&:= \int_0^\infty y^{1-2s}\gamma_b(y-k)\, dy \\
			&\ =\int_0^{\frac{2-b}{1-b}} (y+k)^{1-2s}\gamma_b(y)\, dy\\
			& \ \in  \begin{cases}
				\frac{bk^{1-2s}}{1-b}\LC 1, \LC 1+ \frac{2-b}{k(1-b)}\RC^{1-2s}\RC, \quad \text{if }s\in (0,\frac{1}{2}],\\
				\frac{bk^{1-2s}}{1-b}\LC \LC 1+ \frac{2-b}{k(1-b)}\RC^{1-2s}, 1 \RC, \quad \text{if }s\in (\frac{1}{2},1)
			\end{cases}
		\end{split}
	\end{equation}
	where $I_{b,k}$ depends continuously on $b\in (0,1)$. From the above analysis, it is evident that the values attained by $I_{b,k}$ can range from arbitrarily large to arbitrarily close to 0. By the continuity of for all $k\in \N$, there exist $b_{k,s}\in (0,1)$ such that $I_{b_{s,k},k}=1$. Now, let 
	\[
	\beta_k(y):= \gamma_{b_{k,s}}(y-k) \quad \text{and}\quad R_{k,s}:=k+\frac{1}{1-b_{k,s}},
	\]
	then $\beta_k: (0,\infty)\to [0,1]$ satisfies 
	\begin{equation}
		\begin{split}
			&\supp\LC \beta_k \RC \subset \LC k, R_{k,s}+1\RC , \\
			&\beta_k(y)=b_{k,s}, \text{ for }y\in (k+1,R_{k,s}),\\
			&\left| \p^\ell_y \beta_k(y) \right| \leq C, \\
			&\int_0^\infty y^{1-2s}\beta_k(y)\, dy =1,
		\end{split}
	\end{equation}
	for $\ell=0,1,2$, and for some constant $C>0$ independent of $k\in \N$.

	\section*{Statements and Declarations}
	
	\subsection*{Data availability statement}
	No datasets were generated or analyzed during the current study.
	
	\subsection*{Conflict of Interests} Hereby, we declare there are no conflicts of interest.

	\bigskip

	\noindent\textbf{Acknowledgments.} 
	The author is partially supported by the National Science and Technology Council (NSTC) Taiwan, under the projects 113-2628-M-A49-003 \& 113-2115-M-A49-017-MY3. The author is also a Humboldt research fellow (for experienced researchers) in Germany.

	\bibliography{refs}

\newcommand{\etalchar}[1]{$^{#1}$}
\begin{thebibliography}{VMC{\etalchar{+}}17}

\bibitem[AH13]{arnold2013unique}
Lilian Arnold and Bastian Harrach.
\newblock Unique shape detection in transient eddy current problems.
\newblock {\em Inverse Problems}, 29(9):095004, 19, 2013.

\bibitem[BHHM17]{barth2017detecting}
Andrea Barth, Bastian Harrach, Nuutti Hyv\"onen, and Lauri Mustonen.
\newblock Detecting stochastic inclusions in electrical impedance tomography.
\newblock {\em Inverse Problems}, 33(11):115012, 18, 2017.

\bibitem[BHKS18]{brander2018monotonicity}
Tommi Brander, Bastian Harrach, Manas Kar, and Mikko Salo.
\newblock Monotonicity and enclosure methods for the {$p$}-{L}aplace equation.
\newblock {\em SIAM J. Appl. Math.}, 78(2):742--758, 2018.

\bibitem[BV16]{BV_16}
Claudia Bucur and Enrico Valdinoci.
\newblock {\em Nonlocal diffusion and applications}, volume~20 of {\em Lecture
  Notes of the Unione Matematica Italiana}.
\newblock Springer, [Cham]; Unione Matematica Italiana, Bologna, 2016.

\bibitem[CGRU23]{CGRU2023reduction}
Giovanni Covi, Tuhin Ghosh, Angkana R{\"u}land, and Gunther Uhlmann.
\newblock A reduction of the fractional {C}alder\'on problem to the local
  {C}alder\'on problem by means of the {C}affarelli-{S}ilvestre extension.
\newblock {\em arXiv preprint arXiv:2305.04227}, 2023.

\bibitem[CLL19]{CLL2017simultaneously}
Xinlin Cao, Yi-Hsuan Lin, and Hongyu Liu.
\newblock Simultaneously recovering potentials and embedded obstacles for
  anisotropic fractional {S}chr\"{o}dinger operators.
\newblock {\em Inverse Probl. Imaging}, 13(1):197--210, 2019.

\bibitem[CLR20]{cekic2020calderon}
Mihajlo Cekic, Yi-Hsuan Lin, and Angkana R{\"u}land.
\newblock The {C}alder{\'o}n problem for the fractional {S}chr{\"o}dinger
  equation with drift.
\newblock {\em Cal. Var. Partial Differential Equations}, 59(91), 2020.

\bibitem[CMRU22]{CMRU20}
Giovanni Covi, Keijo M\"{o}nkk\"{o}nen, Jesse Railo, and Gunther Uhlmann.
\newblock The higher order fractional {C}alder\'{o}n problem for linear local
  operators: {U}niqueness.
\newblock {\em Adv. Math.}, 399:Paper No. 108246, 2022.

\bibitem[CS07]{CS_extension}
Luis Caffarelli and Luis Silvestre.
\newblock An extension problem related to the fractional {L}aplacian.
\newblock {\em Comm. Partial Differential Equations}, 32(7-9):1245--1260, 2007.

\bibitem[Dav90]{Davies90}
Edward~Brian Davies.
\newblock {\em Heat kernels and spectral theory}, volume~92 of {\em Cambridge
  Tracts in Mathematics}.
\newblock Cambridge University Press, Cambridge, 1990.

\bibitem[Fei24]{Fei24_TAMS}
Ali Feizmohammadi.
\newblock Fractional {C}alder\'on problem on a closed {R}iemannian manifold.
\newblock {\em Trans. Amer. Math. Soc.}, 377(4):2991--3013, 2024.

\bibitem[FGK{\etalchar{+}}25]{FGKRSU25}
Ali Feizmohammadi, Tuhin Ghosh, Katya Krupchyk, Angkana R{\"u}land, Johannes
  Sj{\"o}strand, and Gunther Uhlmann.
\newblock Fractional anisotropic {C}alder\'{o}n problem with external data.
\newblock {\em arXiv preprint arXiv:2502.00710}, 2025.

\bibitem[FGKU24]{feizmohammadi2021fractional}
Ali Feizmohammadi, Tuhin Ghosh, Katya Krupchyk, and Gunther Uhlmann.
\newblock Fractional anisotropic {C}alder\'on problem on closed {R}iemannian
  manifolds.
\newblock {\em J. Diff. Geom., to appear}, 2024.

\bibitem[FKU24]{FKU2024calder}
Ali Feizmohammadi, Katya Krupchyk, and Gunther Uhlmann.
\newblock Calder\'on problem for fractional {S}chr\"odinger operators on closed
  {R}iemannian manifolds.
\newblock {\em arXiv preprint arXiv:2407.16866}, 2024.

\bibitem[FL24]{FL24}
Ali Feizmohammadi and Yi-Hsuan Lin.
\newblock Entanglement principle for the fractional {L}aplacian with
  applications to inverse problems.
\newblock {\em arXiv preprint arXiv:2412.13118}, 2024.

\bibitem[Gar17]{garde2017comparison}
Henrik Garde.
\newblock Comparison of linear and non-linear monotonicity-based shape
  reconstruction using exact matrix characterizations.
\newblock {\em Inverse Problems in Science and Engineering}, pages 1--18, 2017.

\bibitem[Geb08]{gebauer2008localized}
Bastian Gebauer.
\newblock Localized potentials in electrical impedance tomography.
\newblock {\em Inverse Probl. Imaging}, 2(2):251--269, 2008.

\bibitem[GH18]{griesmaier2018monotonicity}
Roland Griesmaier and Bastian Harrach.
\newblock Monotonicity in inverse medium scattering on unbounded domains.
\newblock {\em SIAM J. Appl. Math.}, 78(5):2533--2557, 2018.

\bibitem[GLX17]{GLX}
Tuhin Ghosh, Yi-Hsuan Lin, and Jingni Xiao.
\newblock The {C}alder\'{o}n problem for variable coefficients nonlocal
  elliptic operators.
\newblock {\em Comm. Partial Differential Equations}, 42(12):1923--1961, 2017.

\bibitem[GRSU20]{GRSU20}
Tuhin Ghosh, Angkana R\"{u}land, Mikko Salo, and Gunther Uhlmann.
\newblock Uniqueness and reconstruction for the fractional {C}alder\'{o}n
  problem with a single measurement.
\newblock {\em J. Funct. Anal.}, 279(1):108505, 42, 2020.

\bibitem[GS17]{garde2017convergence}
Henrik Garde and Stratos Staboulis.
\newblock Convergence and regularization for monotonicity-based shape
  reconstruction in electrical impedance tomography.
\newblock {\em Numerische Mathematik}, 135(4):1221--1251, 2017.

\bibitem[GS19]{garde2019regularized}
Henrik Garde and Stratos Staboulis.
\newblock The regularized monotonicity method: detecting irregular indefinite
  inclusions.
\newblock {\em Inverse Probl. Imaging}, 13(1):93--116, 2019.

\bibitem[GSU20]{GSU20}
Tuhin Ghosh, Mikko Salo, and Gunther Uhlmann.
\newblock The {C}alder\'{o}n problem for the fractional {S}chr\"{o}dinger
  equation.
\newblock {\em Anal. PDE}, 13(2):455--475, 2020.

\bibitem[Har09]{harrach2009uniqueness}
Bastian Harrach.
\newblock On uniqueness in diffuse optical tomography.
\newblock {\em Inverse Problems}, 25(5):055010, 14, 2009.

\bibitem[Har12]{harrach2012simultaneous}
Bastian Harrach.
\newblock Simultaneous determination of the diffusion and absorption
  coefficient from boundary data.
\newblock {\em Inverse Probl. Imaging}, 6(4):663--679, 2012.

\bibitem[HL19]{HL19_monotone1}
Bastian Harrach and Yi-Hsuan Lin.
\newblock Monotonicity-based inversion of the fractional {S}chr\"odinger
  equation {I}. {P}ositive potentials.
\newblock {\em SIAM J. Math. Anal.}, 51(4):3092--3111, 2019.

\bibitem[HL20]{HL20_monotone2}
Bastian Harrach and Yi-Hsuan Lin.
\newblock Monotonicity-based inversion of the fractional {S}ch\"odinger
  equation {II}. {G}eneral potentials and stability.
\newblock {\em SIAM J. Math. Anal.}, 52(1):402--436, 2020.

\bibitem[HLL18]{harrach2018localizing}
Bastian Harrach, Yi-Hsuan Lin, and Hongyu Liu.
\newblock On localizing and concentrating electromagnetic fields.
\newblock {\em SIAM J. Appl. Math.}, 78(5):2558--2574, 2018.

\bibitem[HLU15]{harrach2015combining}
Bastian Harrach, Eunjung Lee, and Marcel Ullrich.
\newblock Combining frequency-difference and ultrasound modulated electrical
  impedance tomography.
\newblock {\em Inverse Problems}, 31(9):095003, 25, 2015.

\bibitem[HLW25]{HLW_2024_log}
Bastian Harrach, Yi-Hsuan Lin, and Tobias Weth.
\newblock The {C}alder\'on problem for the logarithmic {S}chr\"odinger
  equation.
\newblock {\em J. Differential Equations}, 444:113665, 2025.

\bibitem[HM16]{harrach2016enhancing}
Bastian Harrach and Mach~Nguyet Minh.
\newblock Enhancing residual-based techniques with shape reconstruction
  features in electrical impedance tomography.
\newblock {\em Inverse Problems}, 32(12):125002, 21, 2016.

\bibitem[HM18]{harrach2018monotonicity}
Bastian Harrach and Mach~Nguyet Minh.
\newblock Monotonicity-based regularization for phantom experiment data in
  electrical impedance tomography.
\newblock In {\em New Trends in Parameter Identification for Mathematical
  Models}, pages 107--120. Springer, 2018.

\bibitem[HPS19a]{harrach2019dimension}
Bastian Harrach, Valter Pohjola, and Mikko Salo.
\newblock Dimension bounds in monotonicity methods for the {H}elmholtz
  equation.
\newblock {\em SIAM J. Math. Anal.}, 51(4):2995--3019, 2019.

\bibitem[HPS19b]{harrach2018helmholtz}
Bastian Harrach, Valter Pohjola, and Mikko Salo.
\newblock Monotonicity and local uniqueness for the {H}elmholtz equation.
\newblock {\em Anal. PDE}, 12(7):1741--1771, 2019.

\bibitem[HS10]{harrach2010exact}
Bastian Harrach and Jin~Keun Seo.
\newblock Exact shape-reconstruction by one-step linearization in electrical
  impedance tomography.
\newblock {\em SIAM J. Math. Anal.}, 42(4):1505--1518, 2010.

\bibitem[HU13]{harrach2013monotonicity}
Bastian Harrach and Marcel Ullrich.
\newblock Monotonicity-based shape reconstruction in electrical impedance
  tomography.
\newblock {\em SIAM J. Math. Anal.}, 45(6):3382--3403, 2013.

\bibitem[HU15]{harrach2015resolution}
Bastian Harrach and Marcel Ullrich.
\newblock Resolution guarantees in electrical impedance tomography.
\newblock {\em IEEE Trans. Med. Imaging}, 34:1513--1521, 2015.

\bibitem[HU17]{harrach2017local}
Bastian Harrach and Marcel Ullrich.
\newblock Local uniqueness for an inverse boundary value problem with partial
  data.
\newblock {\em Proc. Amer. Math. Soc.}, 145(3):1087--1095, 2017.

\bibitem[KLW22]{KLW2021calder}
Pu-Zhao Kow, Yi-Hsuan Lin, and Jenn-Nan Wang.
\newblock The {C}alder\'{o}n problem for the fractional wave equation:
  uniqueness and optimal stability.
\newblock {\em SIAM J. Math. Anal.}, 54(3):3379--3419, 2022.

\bibitem[Lin22]{lin2020monotonicity}
Yi-Hsuan Lin.
\newblock Monotonicity-based inversion of fractional semilinear elliptic
  equations with power type nonlinearities.
\newblock {\em Calc. Var. Partial Differential Equations}, 61(5):Paper No. 188,
  30, 2022.

\bibitem[Lin24]{lin2024fractional}
Yi-Hsuan Lin.
\newblock The fractional anisotropic {C}alder\'on problem for a nonlocal
  parabolic equation on closed {R}iemannian manifolds.
\newblock {\em arXiv preprint arXiv:2410.17750}, 2024.

\bibitem[LL22]{LL2020inverse}
Ru-Yu Lai and Yi-Hsuan Lin.
\newblock Inverse problems for fractional semilinear elliptic equations.
\newblock {\em Nonlinear Anal.}, 216:Paper No. 112699, 21, 2022.

\bibitem[LL23]{LL2022inverse}
Yi-Hsuan Lin and Hongyu Liu.
\newblock Inverse problems for fractional equations with a minimal number of
  measurements.
\newblock {\em Commun. Anal. Comput.}, 1(1):72--93, 2023.

\bibitem[LL25]{LL25_IPBook}
Yi-Hsuan Lin and Hongyu Liu.
\newblock {\em Inverse {P}roblems for {I}ntegro-differential {O}perators},
  volume 222 of {\em Applied Mathematical Sciences}.
\newblock Springer, Cham, 2025.

\bibitem[LLU23]{LLU2023calder}
Ching-Lung Lin, Yi-Hsuan Lin, and Gunther Uhlmann.
\newblock The {C}alder\'{o}n problem for nonlocal parabolic operators: {A} new
  reduction from the nonlocal to the local.
\newblock {\em arXiv preprint arXiv:2308.09654}, 2023.

\bibitem[LNZ24]{LNZ_Calderon}
Yi-Hsuan Lin, Gen Nakamura, and Philipp Zimmermann.
\newblock The {C}alder\'on problem for the schr\"odinger equation in
  transversally anisotropic geometries with partial data.
\newblock {\em arXiv preprint arXiv:2408.08298}, 2024.

\bibitem[LZ23]{LZ2023unique}
Yi-Hsuan Lin and Philipp Zimmermann.
\newblock Unique determination of coefficients and kernel in nonlocal porous
  medium equations with absorption term.
\newblock {\em arXiv preprint arXiv:2305.16282}, 2023.

\bibitem[LZ24]{LZ2024approximation}
Yi-Hsuan Lin and Philipp Zimmermann.
\newblock Approximation and uniqueness results for the nonlocal diffuse optical
  tomography problem.
\newblock {\em arXiv preprint arXiv:2406.06226}, 2024.

\bibitem[MVVT16]{maffucci2016novel}
Antonio Maffucci, Antonio Vento, Salvatore Ventre, and Antonello Tamburrino.
\newblock A novel technique for evaluating the effective permittivity of
  inhomogeneous interconnects based on the monotonicity property.
\newblock {\em IEEE Transactions on Components, Packaging and Manufacturing
  Technology}, 6(9):1417--1427, 2016.

\bibitem[RO16]{Ros-Oton_16}
Xavier Ros-Oton.
\newblock Nonlocal elliptic equations in bounded domains: a survey.
\newblock {\em Publ. Mat.}, 60(1):3--26, 2016.

\bibitem[RS18]{ruland2018exponential}
Angkana R\"{u}land and Mikko Salo.
\newblock Exponential instability in the fractional {C}alder\'{o}n problem.
\newblock {\em Inverse Problems}, 34(4):045003, 21, 2018.

\bibitem[RS20]{RS17}
Angkana R\"{u}land and Mikko Salo.
\newblock The fractional {C}alder\'{o}n problem: low regularity and stability.
\newblock {\em Nonlinear Anal.}, 193:111529, 56, 2020.

\bibitem[R{\"u}l25]{ruland2023revisiting}
Angkana R{\"u}land.
\newblock Revisiting the {A}nisotropic {F}ractional {C}alder\'on {P}roblem.
\newblock {\em Int. Math. Res. Not. IMRN}, 2025(5):rnaf036, 2025.

\bibitem[SKJ{\etalchar{+}}19]{seo2018learning}
Jin~Keun Seo, Kang~Cheol Kim, Ariungerel Jargal, Kyounghun Lee, and Bastian
  Harrach.
\newblock A learning-based method for solving ill-posed nonlinear inverse
  problems: a simulation study of lung {EIT}.
\newblock {\em SIAM J. Imaging Sci.}, 12(3):1275--1295, 2019.

\bibitem[ST10]{ST10}
Pablo~Ra\'{u}l Stinga and Jos\'{e}~Luis Torrea.
\newblock Extension problem and {H}arnack's inequality for some fractional
  operators.
\newblock {\em Comm. Partial Differential Equations}, 35(11):2092--2122, 2010.

\bibitem[SU87]{sylvester1987global}
John Sylvester and Gunther Uhlmann.
\newblock A global uniqueness theorem for an inverse boundary value problem.
\newblock {\em Ann. of Math. (2)}, 125(1):153--169, 1987.

\bibitem[SUG{\etalchar{+}}17]{su2017monotonicity}
Zhiyi Su, Lalita Udpa, Gaspare Giovinco, Salvatore Ventre, and Antonello
  Tamburrino.
\newblock Monotonicity principle in pulsed eddy current testing and its
  application to defect sizing.
\newblock In {\em Applied Computational Electromagnetics Society
  Symposium-Italy (ACES), 2017 International}, pages 1--2. IEEE, 2017.

\bibitem[TR02]{tamburrino2002new}
Antonello Tamburrino and Guglielmo Rubinacci.
\newblock A new non-iterative inversion method for electrical resistance
  tomography.
\newblock {\em Inverse Problems}, 18(6):1809--1829, 2002.
\newblock Special section on electromagnetic and ultrasonic nondestructive
  evaluation.

\bibitem[TSV{\etalchar{+}}16]{tamburrino2016monotonicity}
Antonello Tamburrino, Zhiyi Sua, Salvatore Ventre, Lalita Udpa, and Satish~S
  Udpa.
\newblock Monotonicity based imang method in time domain eddy current testing.
\newblock {\em Electromagnetic Nondestructive Evaluation (XIX)}, 41:1, 2016.

\bibitem[VMC{\etalchar{+}}17]{ventre2017design}
Salvatore Ventre, Antonio Maffucci, Fran{\c{c}}ois Caire, Nechtan Le~Lostec,
  Antea Perrotta, Guglielmo Rubinacci, Bernard Sartre, Antonio Vento, and
  Antonello Tamburrino.
\newblock Design of a real-time eddy current tomography system.
\newblock {\em IEEE Transactions on Magnetics}, 53(3):1--8, 2017.

\bibitem[ZHS18]{zhou2018monotonicity}
Liangdong Zhou, Bastian Harrach, and Jin~Keun Seo.
\newblock Monotonicity-based electrical impedance tomography for lung imaging.
\newblock {\em Inverse Problems}, 34(4):045005, 25, 2018.

\end{thebibliography}

	\bibliographystyle{alpha}

\end{document}